\documentclass[12pt, a4paper,reqno]{amsart}
\usepackage[margin=1in]{geometry} 
\usepackage{amsfonts, amssymb,amsmath,amsthm}
\usepackage{graphicx}
\usepackage{subcaption}
\usepackage[dvipsnames]{xcolor}

\newtheorem{theorem}{Theorem}
\DeclareMathOperator*{\res}{Res}

\newtheorem{lemma}{Lemma}
\newtheorem{corollary}{Corollary}

\newtheorem*{remark*}{Remark}

\theoremstyle{definition}

\numberwithin{equation}{section}

\begin{document}

\title[The second discrete moment of the Riemann zeta function]{The discrete second moment of mixed derivatives of the Riemann zeta function}
\author[B. Durkan]{Benjamin Durkan}
\address{Department of Mathematics, The University of Manchester, Oxford Road, Manchester, M13 9PL}
\email{benjamin.durkan@postgrad.manchester.ac.uk}
\author[C. Hughes]{Christopher Hughes}
\address{Department of Mathematics, University of York, York, YO10 5GH, United Kingdom}
\email{christopher.hughes@york.ac.uk}
\author[A. Pearce-Crump]{Andrew Pearce-Crump}
\address{School of Mathematics, Fry Building, Woodland Road, Bristol, BS8 1UG, United Kingdom}
\email{andrew.pearce-crump@bristol.ac.uk}
\date{}

\begin{abstract}
We establish the full asymptotic for the discrete second moment of the Riemann zeta function of mixed derivatives evaluated at the zeta zeros, providing both unconditional and conditional error terms. This was first studied by Gonek, where only the leading order asymptotic was given, later extended by Conrey--Snaith and Milinovich to include the lower order terms for the first derivative. We extend the case of the first derivative to all derivatives.
\end{abstract}

\maketitle

\section{Introduction}
The central object of this paper is the discrete second moment of mixed derivatives of the Riemann zeta function, given by the sum
\begin{equation}\label{eq:main_eq}
    I(\mu,\nu):=I(\mu,\nu;T)=\sum_{0<\gamma\le T}\zeta^{(\mu)}(\rho)\zeta^{(\nu)}(1-\rho),
\end{equation}
as $T\to\infty$, where the sum is over non-trivial zeros of zeta $\rho=\beta+i\gamma$, and where $\zeta^{(\mu)}(s)$ denotes the $\mu^{\text{th}}$ derivative of the Riemann zeta function. We will establish a full asymptotic expansion for this sum, with a power-saving error term under the Riemann Hypothesis.

This type of discrete moment was introduced by Gonek \cite{Gonek1984}, who proved for positive integers $\mu,\nu$, a leading order asymptotic of the form \begin{multline}\label{eq:gonek_general}
    \sum_{0<\gamma\le T}\zeta^{(\mu)}(\rho)\zeta^{(\nu)}(1-\rho) \\ =
    (-1)^{\mu+\nu}\left(\frac{1}{\mu+\nu+1}-\frac{1}{(\mu+1)(\nu+1)}\right)\frac{T}{2\pi}\left(\log \frac{T}{2\pi}\right)^{\mu+\nu+2}+O\left(T(\log T)^{\mu+\nu+1}\right)
\end{multline}
as $T\to\infty$.

The special case $\mu=\nu$, called the discrete second moment of zeta, is given under the Riemann Hypothesis by
\begin{equation}\label{eq:gonek_muequalsnu}
    \sum_{0<\gamma\le T}\left|\zeta^{(\nu)}\left(\frac{1}{2}+i\gamma\right)\right|^2=\frac{\nu^2}{(2\nu+1)(\nu+1)^2}\frac{T}{2\pi}\left(\log \frac{T}{2\pi}
\right)^{2\nu+2}+O\left(T(\log T)^{2\nu+1}\right)
\end{equation}
as $T \rightarrow \infty$.

Conrey and Snaith~\cite{conrey2007applicationslfunctionsratiosconjectures} conjectured under the Riemann Hypothesis that for $\varepsilon > 0$ arbitrary and $L=\log t/2\pi$,
\begin{align*}
\sum_{0 < \gamma \leq T} |\zeta'(\rho)|^2 
= \frac{1}{2\pi}\int_1^T \Bigg(&
    \frac{1}{12} L^4
    + \frac{2\gamma_0}{3} L^3
    + \left( \gamma_0^2 - 2\gamma_1 \right) L^2 - \left( 2\gamma_0^3 + 2\gamma_0 \gamma_1 + \gamma_2 \right) L \\
    &+ \left( 2\gamma_0^4 + 2\gamma_0^2 \gamma_1 + 14\gamma_1^2 
    + 8\gamma_0 \gamma_2 + \frac{10\gamma_3}{3} \right)
\Bigg) \ dt + O\left(T^{1/2 + \varepsilon} \right),
\end{align*}
as $T \rightarrow \infty$, where the $\gamma_n$ are the Stieltjes coefficients from the expansion of $\zeta(s)$ around $s=1$,
\begin{equation*}
    \zeta(s)=\frac{1}{s-1}+\gamma_0-\gamma_1(s-1)+\frac{\gamma_2}{2!}(s-1)^2+\cdots+(-1)^n \frac{\gamma_n}{n!} (s-1)^n + \cdots.
\end{equation*}

Milinovich~\cite{micah} proved this conjecture under the assumption of the Riemann Hypothesis, writing the asymptotic as
    \begin{equation}\label{milinovich_theorem}
        \sum_{0<\gamma\le T}\left|\zeta'\left(\frac{1}{2}+i\gamma\right)\right|^2=\frac{T}{2\pi} P_4\left(\log\frac{T}{2\pi}\right)+O\left(T^{\frac{1}{2}+\varepsilon}\right)
    \end{equation}
as $T \rightarrow \infty$, where $P_4(x)$ is a degree four polynomial given by
    \begin{align*}
P_4(x)  &= \frac{1}{12} x^4 \\
      &+ \left( \frac{2\gamma_0 - 1}{3} \right) x^3 \\
      &+ \left( 1 - 2\gamma_0 + \gamma_0^2 - 2\gamma_1 \right) x^2 \\
      &+ \left( -2 + 4\gamma_0 - 2\gamma_0^2 - 2\gamma_0^3 - 10\gamma_0\gamma_1 + 4\gamma_1 - \gamma_2 \right) x \\
      &+ \left( \frac{6 + 6\gamma_0(5\gamma_1 + 4\gamma_2 - 2) + 6\gamma_0^2(\gamma_0 + \gamma_0^2 + 6\gamma_1 + 1) 
            - 12\gamma_1 + 42\gamma_1^2 + 3\gamma_2 + 10\gamma_3}{3} \right)
\end{align*}

\begin{remark*}
    The equivalence of Milinovich's result and the conjecture of Conrey and Snaith follows by performing the integral.
\end{remark*}

Some examples of other generalisations of the results discussed above are shifted second moment results \cite{garunkvstis2025mean, garunks_novikas}, higher moment conjectures \cite{gonek1989negative, Hejhal1989, hughes2000, ng2004}, extreme values of derivatives of zeta \cite{ng2008}, upper and lower bounds \cite{kirila2020, benli2023} on moments, and negative moments \cite{gonek1989negative, garaev2003, heap20222, gao2023, bui2024}.

\section{Statement of Results}

We generalise the results of both \eqref{eq:gonek_general} and \eqref{milinovich_theorem} by proving a full asymptotic expansion for the discrete second moment of mixed derivatives of the Riemann zeta function \eqref{eq:main_eq}, both unconditionally and conditionally under the Riemann Hypothesis.

\begin{theorem}\label{mainthm}
    For positive integers $\mu,\nu$, we have 
    \begin{equation}\label{mainthmeq}
        \sum_{0<\gamma\le T}\zeta^{(\mu)}(\rho)\zeta^{(\nu)}(1-\rho)=\frac{T}{2\pi}\mathcal{P}_{\mu,\nu}\left(\log\frac{T}{2\pi}\right)+O\left(Te^{-C\sqrt{\log T}}\right)
    \end{equation}
    as $T\to\infty$, where $C$ is a positive constant and where $\mathcal{P}_{\mu,\nu}(x)$ is the polynomial of degree $\mu+\nu+2$ given by
    \begin{multline}\label{eq:poly}
        \mathcal{P}_{\mu,\nu}(x) = \sum_{m=0}^{\mu+\nu+2} \sum_{k=0}^{\nu}  (-1)^{\nu}  \binom{\nu}{k} C_1^{(\mu,\nu)}(m,k)x^m\\
        +\sum_{m=0}^{\mu+\nu+2} \sum_{k=0}^{\mu}  (-1)^{\mu}  \binom{\mu}{k} \left(C_1^{(\nu,\mu)}(m,k)+C_2^{(\mu,\nu)}(m,k)\right)x^m,
    \end{multline}
    where
\begin{multline*}
C_1^{(\mu,\nu)}(m,k) = \\
\begin{cases} 
\displaystyle \frac{(-1)^m(\nu-k)}{m!} \sum_{j=0}^{\mu+\nu+1-m}  (-1)^{\mu+\nu-j} \frac{ (\mu+\nu+1-j)! }{(\mu+k+2-j)!} c^{(\mu,k)}_{j} +\\
\displaystyle \qquad + \frac{c^{(\mu,k)}_{\mu+\nu+2-m}}{(k+m-\nu)!} &  m \geq \nu-k \\[3ex]
\displaystyle \frac{(-1)^m(\nu-k)}{m!} \sum_{j=0}^{\mu+k+2}  (-1)^{\mu+\nu-j} \frac{ (\mu+\nu+1-j)! }{(\mu+k+2-j)!} c^{(\mu,k)}_{j}  & m \leq \nu -k-1
\end{cases}
\end{multline*}
where $c^{(\mu,k)}_{j}$ are the Laurent series coefficients around $s=1$ of \begin{equation}
    \frac{\zeta'}{\zeta}(s)\zeta^{(\mu)}(s)\zeta^{(k)}(s)\frac{1}{s}=\sum_{j=0}^{\infty}c^{(\mu,k)}_{j}(s-1)^{-\mu-k-3+j} ,
\end{equation}
    and  where
\begin{multline*}
    C_2^{(\mu,\nu)}(m,k) = \\
\begin{cases} 
\displaystyle  \frac{(-1)^m(\mu-k+1)}{m!} \sum_{j=0}^{\mu+\nu+1-m} (-1)^{\mu+\nu-j}  \frac{(\mu+\nu+1-j)!}{(\nu+k+1-j)!} d_{j}^{(\nu,k)} +  \\
\displaystyle \qquad  +\frac{d_{\mu+\nu+2-m}^{(\nu,k)}}{(k+m-\mu-1)!} & m \geq \mu-k+1 \\[3ex]
\displaystyle  \frac{(-1)^m(\mu-k+1)}{m!} \sum_{j=0}^{\nu+k+1} (-1)^{\mu+\nu-j}  \frac{(\mu+\nu+1-j)!}{(\nu+k+1-j)!} d_{j}^{(\nu,k)}  & m \leq \mu-k
\end{cases}
\end{multline*}
where $d_{j}^{(\nu,k)}$ are the Laurent series coefficients around $s=1$ of 
    \begin{equation*}
        \zeta^{(\nu)}(s)\zeta^{(k)}(s)\frac{1}{s}=\sum_{j=0}^{\infty}d_{j}^{(\nu,k)}(s-1)^{-\nu-k-2+j}.
    \end{equation*}

If one assumes the Riemann Hypothesis, the error term may be replaced with $O\left(T^{\frac{1}{2}+\varepsilon}\right)$ for arbitrary $\varepsilon>0$.
\end{theorem}

We now state several corollaries of this result, with Corollary \ref{cor:higher_derivs_milinovich} clearly following immediately from this theorem and with Corollaries \ref{cor:micah} and \ref{corollary:general} proved in Section \ref{conclusion}.

\begin{corollary}\label{cor:higher_derivs_milinovich}
    Assume the Riemann Hypothesis. For $\nu$ a positive integer, the discrete second moment of zeta for all derivatives is given by
    \begin{multline*}
        \sum_{0<\gamma\le T}\left|\zeta^{(\nu)}\left(\frac{1}{2}+i\gamma\right)\right|^2
        \\=\frac{T}{2\pi} \sum_{m=0}^{2\nu+2} \left( \sum_{k=0}^{\nu} (-1)^{\nu} \binom{\nu}{k}  \left(2C_1^{(\nu,\nu)}(m,k)+C_2^{(\nu,\nu)}(m,k)\right) \right) \left(\log\frac{T}{2\pi}\right)^m + O\left(T^{\frac{1}{2}+\varepsilon}\right),
    \end{multline*}
    where $C_1^{(\nu,\nu)}(m,k)$ and $C_2^{(\nu,\nu)}(m,k)$ are defined in Theorem \ref{mainthm}.

\end{corollary}

\begin{corollary} \label{cor:micah}
    Assume the Riemann Hypothesis. In the case $\mu=\nu=1$, \eqref{mainthmeq} recovers the polynomial $P_4(x)$ in \eqref{milinovich_theorem}.
\end{corollary}

\begin{corollary}\label{corollary:general}
    For $\mu,\nu$ positive integers, the leading order coefficient of $\mathcal{P}_{\mu,\nu}(x)$ agrees with the leading order given in \eqref{eq:gonek_general}.
\end{corollary}

As we see in Milinovich's polynomial $P_4(x)$ in \eqref{milinovich_theorem}, the coefficients quickly become unwieldy. We give an example of our result for discrete second moment of the second derivative in Appendix~\ref{app:2ndDeriv} by explicitly writing out the polynomial for the second derivative. Finally, in Appendix~\ref{app:graphs} we demonstrate the theorem by plotting the graphs for the first and second derivatives.

\section*{Acknowledgments}
This forms part of the first author's MSc by Research thesis \cite{Durkan} from the University of York. The third author acknowledges support from the Heilbronn Institute for Mathematical Research.

\section{Brief outline of the proof}

For $c=1+\frac{1}{\log T}$ and $\mathcal{R}$ the rectangular contour with vertices $c+i$, $c+i T$, $1-c+iT$, and $1-c+i$,  we may use Cauchy's theorem to write
\begin{align}
    I(\mu,\nu)&=\sum_{0<\gamma\le T}\zeta^{(\mu)}(\rho)\zeta^{(\nu)}(1-\rho) \notag\\
    &=\frac{1}{2\pi i}\oint_{\mathcal{R}}\frac{\zeta'}{\zeta}(s)\zeta^{(\mu)}(s)\zeta^{(\nu)}(1-s)\ ds \notag\\
    &=\frac{1}{2\pi i}\left(\int_{c+i}^{c+iT}+\int_{c+iT}^{1-c+iT}+\int_{1-c+iT}^{1-c+i}+\int_{1-c+i}^{c+i}\right)\frac{\zeta'}{\zeta}(s)\zeta^{(\mu)}(s)\zeta^{(\nu)}(1-s)\ ds \notag\\
    &= I_1(\mu,\nu)+I_2(\mu,\nu)+I_3(\mu,\nu)+I_4(\mu,\nu). \label{eq:defI3}
\end{align}

It is immediate that $I_4(\mu,\nu)=O(1)$ since the integral has finite length and bounded integrand. Furthermore, as shown by Gonek \cite{Gonek1984}, we have $I_2(\mu,\nu)=O\left(T^{\frac{1}{2}+\varepsilon}\right)$ with the harmless restriction 
that $T$ lies a distance $\gg 1/\log T$ from a zero ordinate $\gamma$. We shall use this restriction without loss of generality throughout our proof. From this it follows that 
\begin{equation}\label{eq:I_I1_I3}
    I(\mu,\nu)=I_1(\mu,\nu)+I_3(\mu,\nu)+O\left(T^{\frac{1}{2}+\varepsilon}\right).
\end{equation}

In Section \ref{sect:right_vertical} we deal with the right-hand vertical segment $I_1(\mu,\nu)$. The starting point is by introducing a functional equation which allows us to write $\zeta^{(\nu)}(1-s)$ in terms of $\zeta^{(k)}(s)$ for $k$ between $0$ and $\nu$. As a consequence we may write $I_1(\mu,\nu)$ as an integral of an absolutely convergent Dirichlet series. We then use the method of stationary phase to rewrite this integral as the sum
\begin{equation*}
    \sum_{n_1 n_2 n_3\le\frac{T}{2\pi}}\Lambda(n_1) (\log n_2)^{\mu} (\log n_3)^k (\log n_1 n_2 n_3)^{\nu-k},
\end{equation*}
plus a small error that is subsumed by our error term written above.

We evaluate this sum without the $(\log n_1 n_2 n_3)^{\nu-k}$ factor through Perron's formula. This gives the sum as the residue of a Dirichlet series at $s=1$, which in turn is a polynomial of degree $\mu+\nu+2$. Finally we reinsert the logarithmic term via partial summation.

In Section \ref{sect:left_vertical} we handle the left-hand segment $I_3(\mu,\nu)$. We use another functional equation to relate the left vertical segment of our integral to the right vertical segment, and then follow similar methods to those used in Section \ref{sect:right_vertical} to establish the asymptotics for the left vertical segment. Combining these two contours, together with the error term that we carry through the proof, gives Theorem \ref{mainthm}.

Finally, as mentioned in the introduction,  in Section \ref{conclusion} we prove some corollaries of our result.

\section{The right-hand vertical segment}\label{sect:right_vertical}
\subsection{Initial manipulations} \hfill

We start by writing $I_1(\mu,\nu)$ in a form more amenable to analysis. Recall that \begin{equation*}
    I_1(\mu,\nu)=\frac{1}{2\pi}\int_1^T\frac{\zeta'}{\zeta}(c+it)\zeta^{(\mu)}(c+it)\zeta^{(\nu)}(1-c-it)\ dt.
\end{equation*}
We apply the functional equation for $\zeta^{(\nu)}(1-c-it)$ so we can write the resulting expression in terms of convergent Dirichlet series.

\begin{lemma}\cite[Lemma 4]{discreteMVT}\label{logderiv}
    For $s=\sigma+it$, with $\sigma\ge 1$ and $t\ge 1$ we have 
    \begin{equation}\label{logderiveq}
        \zeta^{(\nu)}(1-s)=(-1)^{\nu}\chi(1-s)\sum_{k=0}^{\nu} \binom{\nu}{k}\left(\log\frac{t}{2\pi}\right)^{\nu-k} \zeta^{(k)}(s)+O\left(t^{\sigma-\frac{3}{2}}(\log t)^{\nu} \right),
    \end{equation}
    where $\chi(s)$ is the factor from the functional equation $\zeta(s)=\chi(s)\zeta(1-s)$.
\end{lemma}

Substituting \eqref{logderiveq} into our expression for $I_1(\mu,\nu)$ gives
\begin{multline*}
    I_1(\mu,\nu)\\
    =\frac{(-1)^{\nu}}{2\pi}\sum_{k=0}^{\nu} \binom{\nu}{k} \int_1^T\frac{\zeta'}{\zeta}(c+it)\zeta^{(\mu)}(c+it)\zeta^{(k)}(c+it)\chi(1-c-it)\left( \log\frac{t}{2\pi}\right)^{\nu-k} \ dt \\
    +O\left(T^{\frac{1}{2}+\varepsilon} \right).
\end{multline*}

Each term in the integrand can be expressed as a Dirichlet series since $\Re(s)=c>1$, namely 
\begin{equation*}
    \frac{\zeta'}{\zeta}(s)=-\sum_{n=1}^{\infty}\frac{\Lambda(n)}{n^s} \text{ and }
    \zeta^{(\mu)}(s)=(-1)^{\mu}\sum_{n=1}^{\infty}\frac{(\log n)^{\mu}}{n^s},
\end{equation*}
where $\Lambda(n)$ is the von Mangoldt function.

It follows that, by multiplying the Dirichlet series in the integral above together, we can write
\begin{align*}
    \frac{\zeta'}{\zeta}(s)\zeta^{(\mu)}(s)\zeta^{(k)}(s) &= (-1)^{\mu+k+1}\left(\sum_{n_1=1}^{\infty}\frac{\Lambda(n_1)}{n_1^s}\right)\left(\sum_{n_2=1}^{\infty}\frac{(\log n_2)^{\mu}}{n_2^s}\right)\left(\sum_{n_3=1}^{\infty}\frac{(\log n_3)^k}{n_3^s}\right) \\
    &=\sum_{n=1}^{\infty}\frac{A_n^{(\mu,k)}}{n^s}
\end{align*}
where
\begin{equation}\label{eq:Def_A_k_mu_k}
    A_n^{(\mu,k)}:=(-1)^{\mu+k+1}\sum_{n_1n_2n_3=n}\Lambda(n_1)(\log n_2)^{\mu}(\log n_3)^k .
\end{equation}

The integral $I_1(\mu, \nu)$ can then be written as
\begin{multline*}
    I_1(\mu,\nu)\\
    =\frac{(-1)^{\nu}}{2\pi} \sum_{k=0}^{\nu} \binom{\nu}{k} \int_1^T\chi(1-c-it)\left( \log\frac{t}{2\pi}\right)^{\nu-k} \left( \sum_{n=1}^{\infty} \frac{A_n^{(\mu,k)}}{n^{c+it}}\right) \ dt+O\left(T^{\frac{1}{2}}(\log T)^{\nu}\right).
\end{multline*}

We now use a lemma due to Gonek \cite[Lemma 5]{Gonek1984} which comes from the method of stationary phase which will permit us to write this integral as a sum, which we shall subsequently evaluate via the method of Perron.

\begin{lemma}\label{5.1.2}
    Let $\{a_n\}_{n=1}^{\infty}$ be a sequence of complex numbers that satisfy $a_n\ll n^{\varepsilon}$ for arbitrary $\varepsilon>0$. Let $c=1+\frac{1}{\log T}$, and let $m$ be a non-negative integer. Then for $T\ge 1$ we have 
    \begin{equation*}
        \frac{1}{2\pi}\int_1^T\ \chi(1-c-it) \left(\log\frac{t}{2\pi}\right)^m \left( \sum_{n=1}^{\infty} \frac{a_n}{n^{c+it}}\right) \ dt = \sum_{1\le n\le\frac{T}{2\pi}}a_n(\log n)^m+O\left(T^{c-\frac{1}{2}}(\log T)^m\right).
    \end{equation*}
\end{lemma}

Then by Lemma \ref{5.1.2}, we can rewrite $I_1(\mu,\nu)$ as the sum
\begin{equation} \label{eq:I1_as_short_DS}
    I_1(\mu,\nu)= \sum_{k=0}^{\nu} (-1)^{\nu} \binom{\nu}{k} 
\sum_{n\le\frac{T}{2\pi}}A_n^{(\mu,k)}(\log n)^{\nu-k}+O\left(T^{\frac{1}{2}+\varepsilon}\right).
\end{equation}

\subsection{Evaluating the inner sum without the logarithm}\label{subsection: sum=residue} \hfill

Throughout this section we  we set $Y = \frac{T}{2\pi}$ for notational convenience.

We now evaluate the sum in \eqref{eq:I1_as_short_DS}. One could use Perron's formula, relying on the Dirichlet series expansion
\[
\sum_{n=1}^\infty \frac{A_n^{(\mu,k)}(\log n)^{\nu-k}}{n^s} = (-1)^{\nu-k} \frac{d^{\nu-k}}{ds^{\nu-k}} \left(\frac{\zeta'}{\zeta}(s)\zeta^{(\mu)}(s)\zeta^{(k)}(s) \right).
\]
However it turns out to be simpler and more convenient to use Perron's formula to calculate $\sum_{n\le Y} A_n^{(\mu,k)}$ without any logarithms, and then use partial summation to reinsert them later. 

We start by using a truncated version of Perron's formula.

\begin{lemma}\label{lem11DMVT}
    Let $A_n^{(\mu,k)}$ be given by \eqref{eq:Def_A_k_mu_k}. For $2\le V\le Y$, as $Y\to\infty$, 
    \begin{equation*}
        \sum_{n\le Y}A_n^{(\mu,k)}=\frac{1}{2\pi i}\int_{c-iV}^{c+iV}\frac{\zeta'}{\zeta}(s)\zeta^{(\mu)}(s)\zeta^{(k)}(s)\frac{Y^s}{s} \ ds+R_1(Y,V),
    \end{equation*}
    where 
    \begin{equation*}
        R_1(Y,V)= O\left(\frac{Y}{V}(\log Y)^{\mu+k+3}\right).
    \end{equation*}
\end{lemma}

\begin{proof}
Bounding $R_1(Y,V)$ is all we need to do to prove the lemma, since the integral comes directly from Perron's formula (and the truncation gives the error term). By \cite[Cor 5.3]{Montgomery_Vaughan_2006} we can bound the remainder as 
\begin{align}
    R_1(Y,V)&\ll\sum_{\frac{Y}{2}<n<2Y}|A_n^{(\mu,k)}|\min\left(1,\frac{Y}{V|Y-n|}\right)+\frac{4^c+Y^c}{V}\sum_{n\ge 1}\frac{|A_n^{(\mu,k)}|}{ n^c}\notag\\
    &=A+B\label{perron2},
\end{align}
say. To evaluate the error, we start by giving an upper bound for $A_n^{(\mu,k)}$. Note that $\Lambda(n)\leq \log n$  with equality holding if and only if $n$ is prime. We bound $A_n^{(\mu,k)}$ crudely by
\begin{equation*}
    \left|A_n^{(\mu,k)}\right| \leq \sum_{n_1n_2n_3=n}(\log n)(\log n)^{\mu}(\log n)^{k}= (\log n)^{\mu+k+1}d_3(n),
\end{equation*}
where $d_3(n)$ is the $3$-fold divisor function. Bounding $A$ in \eqref{perron2} by the maximum of $A_n^{(\mu,k)}$ and using the notation $\ell=|Y-m|$, 
\begin{align*}
    A&\ll (\log Y)^{\mu+k+1}\sum_{\frac{Y}{2}<m<2Y}d_3(m)\min\left(1,\frac{Y}{V\ell}\right)\\&\ll(\log Y)^{\mu+k+1}\sum_{\ell\le Y}\frac{Y}{V}\frac{1}{\ell}d_3(\ell)\\&\ll\frac{Y}{V}(\log Y)^{\mu+k+4},
\end{align*}
using the fact that $\sum_{\ell\le Y}\frac{d_3(\ell)}{\ell}\ll (\log Y)^3$ (see, for example, \cite[p.43]{Montgomery_Vaughan_2006}). 

Next we turn to bounding $B$ in \eqref{perron2}. Since $c=1+\frac{1}{\log T}$ the Dirichlet series converges, and so \begin{equation*}
    B=\frac{4^c+Y^c}{V}\left|\frac{\zeta'}{\zeta}(c)\zeta^{(\mu)}(c)\zeta^{(k)}(c)\right|\ll\frac{Y}{V}(\log Y)^{\mu+k+3},
\end{equation*}
using the fact that $Y = \frac{T}{2\pi}$ and also that 
\begin{align}
    \frac{\zeta'}{\zeta}\left(1+\frac{1}{\log T}\right) &\ll \log T, \label{eq:LogDerivZetaNearPole}\\
\intertext{and}
    \zeta^{(\mu)}\left(1+\frac{1}{\log T}\right) &\ll (\log T)^{\mu+1}\label{eq:DerivZetaNearPole}
\end{align}
as required.
\end{proof}

Next we write the integral in Lemma \ref{lem11DMVT} in terms of the residue of the integrand at $s=1$. We obtain an error term from completing the contour into a rectangle which we will deal with in Lemma \ref{lem12DMVT}. We remark that this lemma is the only place where we get different errors depending on whether we assume the Riemann Hypothesis or not, so we shall consider both cases separately.

\begin{lemma}\label{lem12DMVT}
    For $Y=\frac{T}{2\pi}$, as $V,Y\to\infty$, we have \begin{equation*}
        \frac{1}{2\pi i}\int_{c-iV}^{c+iV}\frac{\zeta'}{\zeta}(s)\zeta^{(\mu)}(s)\zeta^{(k)}(s)\frac{Y^s}{s}\ ds=\res_{s=1}\left(\frac{\zeta'}{\zeta}(s)\zeta^{(\mu)}(s)\zeta^{(k)}(s)\frac{Y^s}{s}\right)+\mathcal{E}_1(Y,V),
    \end{equation*}
    where $\mathcal{E}_1(\mu,k;Y)$ is given by \begin{enumerate}
        \item $\mathcal{E}_1(\mu,k;Y) = \left( Ye^{-C\sqrt{\log Y}}\right)$ unconditionally for a positive constant $C$
        \item $\mathcal{E}_1(\mu,k;Y) = \left( Y^{\frac{1}{2}+\varepsilon}\right)$ under the Riemann Hypothesis.
    \end{enumerate}
\end{lemma}

\begin{proof}[Proof of Lemma \ref{lem12DMVT}]
We consider the positively oriented rectangular contour with vertices at $c\pm iV, c'\pm iV$, where $\frac{1}{2} < c'<1$ is large enough that no zeros of zeta with $|\gamma| \leq V$ lie on or to the right of the $c'$ line. Therefore the only pole contained within this contour is at $s=1$, and by Cauchy's Residue Theorem,
\begin{multline*}
    \res_{s=1}\left(\frac{\zeta'}{\zeta}(s)\zeta^{(\mu)}(s)\zeta^{(k)}(s)\frac{Y^s}{s}\right) \\
    =\frac{1}{2\pi i}\left(\int_{c-iV}^{c+iV}+\int_{c+iV}^{c'+iV}+\int_{c'+iV}^{c'-iV}+\int_{c'-iV}^{c-iV}\right)\frac{\zeta'}{\zeta}(s)\zeta^{(\mu)}(s)\zeta^{(k)}(s)\frac{Y^s}{s}\ ds.
\end{multline*}

By rearranging this expression and switching the orientation of some of the integrals we have that the desired integral  is equal to
\begin{multline*}
    \frac{1}{2\pi i}\int_{c-iV}^{c+iV}\frac{\zeta'}{\zeta}(s)\zeta^{(\mu)}(s)\zeta^{(k)}(s)\frac{Y^s}{s} \ ds = \res_{s=1}\left(\frac{\zeta'}{\zeta}(s)\zeta^{(\mu)}(s)\zeta^{(k)}(s)\frac{Y^s}{s}\right)
    \\+\frac{1}{2\pi i}\left(\int_{c'+iV}^{c+iV}+\int_{c'-iV}^{c'+iV}-\int_{c'-iV}^{c-iV}\right)\frac{\zeta'}{\zeta}(s)\zeta^{(\mu)}(s)\zeta^{(k)}(s)\frac{Y^s}{s} \ ds.
    \end{multline*}

The result will follow from finding appropriate bounds for the three line integrals. We consider the conditional and unconditional cases separately as they involve different values of $c'$ and different bounds for zeta.

\subsubsection*{The unconditional case}\hfill

As noted in Titchmarsh \cite[p.54]{titchmarsh}, there exists some absolute constant $C>0$ such that for $c'=1-\frac{C}{\log V}$, any zero of $\zeta(s)$ lies a distance $\gg \frac{1}{\log V}$ away from the line between $c'-iV$ and $c'+iV$. 

Therefore, on the two horizontal pieces we may apply the bound of Gonek~\cite[\S 2]{Gonek1984} which applies uniformly for $-1 \leq \sigma \leq 2$
\[
\frac{\zeta'}{\zeta}(\sigma \pm iV) \ll (\log V)^2 .
\]
We also use the convexity bounds inside and to the right of the critical strip
\begin{equation} \label{eq:convexity}
\zeta^{(\mu)} (\sigma \pm iV) \ll
\begin{cases}
V^{\frac{1}{2}(1- \sigma)} (\log V)^{\mu+1} &\text{ if $ 0 \leq \sigma \leq 1$}\\
(\log V)^{\mu+1} &\text{ if $\sigma \geq 1$,}
\end{cases}
\end{equation}
which may be established by using Ivic's bound for $\nu=1$ \cite{ivic} and applying Cauchy's estimate for derivatives of analytic functions around a disc of radius $1/\log V$ with centre at $\sigma+iV$.

With these bounds we deduce that
\begin{equation*}
    \frac{1}{2\pi i}\int_{c'\pm iV}^{c\pm iV}\frac{\zeta'}{\zeta}(s)\zeta^{(\mu)}(s)\zeta^{(k)}(s)\frac{Y^s}{s}\ d\sigma\ll (\log V)^{\mu+k+4}\frac{Y^c}{V}(c-c')\ll \frac{Y}{V}(\log V)^{\mu+k+3},
\end{equation*}
where we have used the fact that $c-c'\ll \frac{1}{\log V}$ and $Y^c = \ll Y$.

Now we turn to the vertical segment of the contour. Away from the pole, we can use the same bounds as above, and for $t\approx 0$ we bound the terms in the integral by the appropriate Laurent expansions, \eqref{eq:LogDerivZetaNearPole} and \eqref{eq:DerivZetaNearPole}. Therefore on the line $s=c'+it$, $-V\le t\le V$, we can bound the integral by 
\begin{equation*}
Y^{c'} (\log V)^{\mu+k+4} + \int_{1}^V (\log V)^{\mu+k+4} \frac{Y^{c'}}{t}\ dt \ll Y^{c'}(\log V)^{\mu+k+5}.
\end{equation*}

Since $Y^{c'}=Y\exp\left(-\frac{C\log Y}{\log V}\right)$ this completes the proof of the unconditional case.

\subsubsection*{The conditional case}\hfill

 Under the assumption of the Riemann Hypothesis, we can both shift the line of integration further to the left without crossing any zeros, and also use better (conditional) bounds on zeta and its derivatives. Specifically we shift the contour to just to the right of the critical line, letting $c'=\frac{1}{2}+\frac{1}{\log V}$, and we use $\zeta^{(n)}(s)\ll t^{\varepsilon}$ for arbitrary $\varepsilon>0$.

At any point on the vertical segment the closest zero is $\ge \frac{1}{\log V}$ away, so we have $\frac{\zeta'}{\zeta}(s)\ll (\log V)^2$ and so  the vertical integral is bounded by
\begin{equation*}
Y^{\frac{1}{2}} \exp\left(\frac{\log Y}{\log V}\right) V^\varepsilon
\end{equation*}
for arbitrary $\varepsilon>0$.

The horizontal segments are dealt with similarly to before, only this time the length of integration is bounded and we will bound terms like $\log V$ by $V^\varepsilon$ since that's sufficient for our purposes. The two horizontal pieces are bounded by
\begin{equation*}
    \frac{Y}{V} V^{\varepsilon}
\end{equation*}
and this completes the proof in the conditional case.
\end{proof}

\subsubsection*{Completion of the proof of Lemma~\ref{lem12DMVT}}\hfill

Using Lemma \ref{lem11DMVT}, thus far in Lemma \ref{lem12DMVT} we have shown that 
\begin{equation*}
    \sum_{n\le Y}A_n^{(\mu,k)}=\res_{s=1}\left(\frac{\zeta'}{\zeta}(s)\zeta^{(\mu)}(s)\zeta^{(k)}(s)\frac{Y^s}{s}\right)+R_1(Y,V)+\mathcal{E}_1(V,Y).
\end{equation*}
where we now pick an optimal $V$ in terms of $Y$, depending on whether we are bounding the error terms conditionally or unconditionally.

In the unconditional case, we can choose $V=\exp(\sqrt{C\log Y})$ to optimise the error terms. This then gives an error term of 
\begin{equation*}
 R_1(Y,V)+\mathcal{E}_1(V,Y)   \ll Ye^{-\sqrt{C\log Y}}(\log Y)^{\frac{\mu+k+5}{2}}\ll Y\exp(-\Tilde{C}\sqrt{\log Y})
\end{equation*}
for some $\Tilde{C}>0$.

In the conditional case, we take $V=Y$ and find 
\[
R_1(Y,V)+\mathcal{E}_1(V,Y) \ll Y^{\frac{1}{2} + \varepsilon},
\]
for arbitrary $\varepsilon>0$, as stated in Lemma~\ref{lem12DMVT}.

We now compute the residue term. In the following lemma we will write the residue in terms of the Laurent coefficients of our triple Dirichlet series given in \eqref{eq:Def_A_k_mu_k} around $s=1$.

\begin{lemma}\label{bccoeffs}
    We have \begin{equation}\label{bccoeffs_eq}
        \mathrm{Res}_{s=1}\left(\frac{\zeta'}{\zeta}(s)\zeta^{(\mu)}(s)\zeta^{(k)}(s)\frac{Y^s}{s}\right)=Y\sum_{j=0}^{\mu+k+2}\frac{c^{(\mu,k)}_{j}}{(\mu+k+2-j)!}(\log Y)^{\mu+k+2-j},
    \end{equation}
     where $c^{\mu,k}_{j}$ are the Laurent series coefficients around $s=1$ of \begin{equation}\label{ccoeffs}
    \frac{\zeta'}{\zeta}(s)\zeta^{(\mu)}(s)\zeta^{(k)}(s)\frac{1}{s}=\sum_{j=0}^{\infty}c^{(\mu,k)}_{j}(s-1)^{-\mu-k-3+j}.
\end{equation}
\end{lemma}
\begin{proof}
The function $\frac{\zeta'}{\zeta}(s)\zeta^{(\mu)}(s)\zeta^{(k)}(s)\frac{1}{s}$ has a pole of order $\mu+k+3$ at $s=1$. We write this expansion as the series given in \eqref{ccoeffs}.

The expansion of $Y^s$ about $s=1$ is 
\begin{equation*}
    Y^s=Y\left(1+(s-1)\log Y+\frac{(s-1)^2}{2!}(\log Y)^2+\cdots+\frac{(s-1)^k}{k!}(\log Y)^k+\cdots\right).
\end{equation*}
and so the residue of $\frac{\zeta'}{\zeta}(s)\zeta^{(\mu)}(s)\zeta^{(k)}(s)\frac{Y^s}{s}$ is given by the sum of 
\begin{equation*}
    \frac{c^{(\mu,k)}_{j}}{(\mu+k+2-j)!} (\log Y)^{\mu+k+2-j}
\end{equation*}
for each $j=0,\dots,\mu+k+2$.
\end{proof}

\subsection{Reinserting the logarithm in the inner sum}\label{subsection:logarithm_inner_sum} \hfill

Throughout this section we  we set $Y = \frac{T}{2\pi}$ for notational convenience.

\begin{lemma}\label{lemma_messy_binomial}
    For $0\leq k \leq \nu$, we have that 
    \begin{equation*}
        \sum_{n\le Y}A_n^{(\mu,k)}(\log n)^{\nu-k}=Y\sum_{m=0}^{\mu+\nu+2}C_1^{(\mu,\nu)}(m,k)(\log Y)^m+O\left(Ye^{-C\sqrt{\log Y}}\right),
    \end{equation*}
where for $ m \geq \nu-k$,
\begin{equation*}
C_1^{(\mu,\nu)}(m,k) = \frac{(-1)^m(\nu-k)}{m!}\sum_{j=0}^{\mu+\nu+1-m} (-1)^{\mu+\nu-j} \frac{(\mu+\nu+1-j)! }{(\mu+k+2-j)!} c^{(\mu,k)}_{j} + 
\frac{c^{(\mu,k)}_{\mu+\nu+2-m}}{(k+m-\nu)!}
\end{equation*}
and for $m < \nu-k $
\begin{equation*}
C_1^{(\mu,\nu)}(m,k) = \frac{(-1)^m(\nu-k)}{m!}\sum_{j=0}^{\mu+k+2} (-1)^{\mu+\nu-j} \frac{(\mu+\nu+1-j)! }{(\mu+k+2-j)!} c^{(\mu,k)}_{j} 
\end{equation*}
\end{lemma}

\begin{remark*}
The differences in the two expressions for $C_1^{(\mu,\nu)}(m,k)$ are the extra term for larger $m$, and the different upper limit on the $j$-sums. Those two upper limits are the same when $m = \nu-k-1$. Finally, when $m=\mu+\nu+2$, which is the largest it can be, the $j$-sum is empty.
\end{remark*}

\begin{proof}
We have shown using Lemmas \ref{lem11DMVT}, \ref{lem12DMVT} and \ref{bccoeffs} that
\begin{equation*}
    \sum_{n\le Y}A_n^{(\mu,k)}=Y\sum_{j=0}^{\mu+k+2}\frac{c^{(\mu,k)}_{j}}{(\mu+k+2-j)!}(\log Y)^{\mu+k+2-j}+O\left(Ye^{-C\sqrt{\log Y}}\right).
\end{equation*}
By partial summation we have
\begin{equation}\label{eq:partialSummation}
    \sum_{n \le Y} A_n^{(\mu,k)} (\log n)^{\nu-k} = \left(\sum_{n\le Y}A_n^{(\mu,k)}\right)(\log Y)^{\nu-k}-(\nu-k)\int_1^Y \left(\sum_{n\le t}A_n^{(\mu,k)}\right)\frac{(\log t)^{\nu-k-1}}{t}\ dt.
\end{equation}

Clearly
\begin{multline*}
\left(\sum_{n\le Y}A_n^{(\mu,k)}\right)(\log Y)^{\nu-k} = Y \sum_{j=0}^{\mu+k+2} \frac{c^{(\mu,k)}_{j}}{(\mu+k+2-j)!} (\log Y)^{\mu+\nu+2-j}\\
+O\left((\log Y)^{\nu-k}Ye^{-C\sqrt{\log Y}}\right),
\end{multline*}
where $C>0$ isn't necessarily the same constant throughout. Relabeling the sum so $m = \mu+\nu+2-j$ (to help make the power of $\log Y$ clear), the first piece in \eqref{eq:partialSummation} is
\begin{equation}\label{eq:partial1}
\sum_{n\le Y}A_n^{(\mu,k)} (\log Y)^{\nu-k} = Y\sum_{m=\nu-k}^{\mu+\nu+2} \frac{c^{(\mu,k)}_{\mu+\nu+2-m}}{(m+k-\nu)!}(\log Y)^{m}
+O\left(Ye^{-C\sqrt{\log Y}}\right).
\end{equation}

Next, note that the second term in \eqref{eq:partialSummation} is
\begin{equation}\label{eq:partial2}
(\nu-k)\sum_{j=0}^{\mu+k+2}\frac{c^{(\mu,k)}_{j}}{(\mu+k+2-j)!}\int_1^Y(\log t)^{\mu+\nu+1-j}\ dt +O\left(Ye^{-C\sqrt{\log Y}}\right).
\end{equation}
Note that, by repeated integration by parts,
\begin{equation}\label{usefulintegral}
    \int_1^Y (\log t)^n\ dt = (-1)^n n! Y \sum_{m=0}^n \frac{(-1)^m (\log Y)^m}{m!} + O(1),
\end{equation}
so using \eqref{usefulintegral} with $n=\mu+\nu+1-j$ we have that  
\begin{multline*}
   \int_1^Y \left( \sum_{n\le t} A_n^{(\mu,k)} \right) \frac{(\log t)^{\nu-k-1}}{t} \ dt \\
   = \sum_{j=0}^{\mu+k+2}\frac{c^{(\mu,k)}_{j}}{(\mu+k+2-j)!}(-1)^{\mu+\nu+1-j}(\mu+\nu+1-j)!Y\sum_{m=0}^{\mu+\nu+1-j}\frac{(-1)^m(\log Y)^m}{m!}\\
    +O\left(Ye^{-C\sqrt{\log Y}}\right).
\end{multline*}
Next swap the order of summation in the double sum to give
\begin{multline}\label{eq:partial3}
   Y \sum_{m=0}^{\nu-k-1}  \frac{(-1)^m(\log Y)^m}{m!}  \sum_{j=0}^{\mu+k+2}  \frac{c^{(\mu,k)}_{j}}{(\mu+k+2-j)!}(-1)^{\mu+\nu+1-j}(\mu+\nu+1-j)!\\
   +Y \sum_{m=\nu-k}^{\mu+\nu+1}  \frac{(-1)^m(\log Y)^m}{m!} \sum_{j=0}^{\mu+\nu+1-m}  \frac{c^{(\mu,k)}_{j}}{(\mu+k+2-j)!}(-1)^{\mu+\nu+1-j}(\mu+\nu+1-j)!
\end{multline}

Combining the pieces \eqref{eq:partial1}, \eqref{eq:partial2}, and \eqref{eq:partial3} gives
\begin{multline}\label{partial_sum_A}
\sum_{n\le Y}A_n^{(\mu,k)}(\log n)^{\nu-k}= Y\sum_{m=\nu-k}^{\mu+\nu+2} \frac{c^{(\mu,k)}_{\mu+\nu+2-m}}{(m+k-\nu)!}(\log Y)^{m} \\
+(\nu-k) Y \sum_{m=0}^{\mu+\nu+1}  \frac{(-1)^m(\log Y)^m}{m!}  \sum_{j=0}^{\min\{\mu+\nu+1-m \, , \, \mu+k+2\}} (-1)^{\mu+\nu-j} \frac{(\mu+\nu+1-j)!}{(\mu+k+2-j)!} c^{(\mu,k)}_{j}\\
+O\left(Ye^{-C\sqrt{\log Y}}\right)
\end{multline}
for some $C>0$, as required.
\end{proof}

Finally, since $I_1(\mu,\nu)$ is equal to 
\begin{equation*}
    \sum_{k=0}^{\nu} (-1)^{\nu} \binom{\nu}{k} \sum_{n\le\frac{T}{2\pi}}A_n^{(\mu,k)}(\log n)^{\nu-k}+O\left(T^{\frac{1}{2}+\varepsilon}\right),
\end{equation*}
we have by Lemma \ref{lemma_messy_binomial} that
\begin{equation}\label{I1poly_simplified}
    I_1(\mu,\nu)=\frac{T}{2\pi} \sum_{m=0}^{\mu+\nu+2} \sum_{k=0}^{\nu} (-1)^{\nu} \binom{\nu}{k}  C_1^{(\mu,\nu)}(m,k) \left(\log\frac{T}{2\pi}\right)^m+O\left(Te^{-C\sqrt{\log T}}\right)
\end{equation}
with the $C_1^{\mu,\nu}(m,k)$ given in the statement of Lemma~\ref{lemma_messy_binomial}.

\section{The left-hand vertical segment}\label{sect:left_vertical}
The evaluation of the left-hand segment $I_3(\mu,\nu)$ follows closely that that of $I_1(\mu,\nu)$, so in this section we give the overview and highlight where the differences between the two calculations lie.

\subsection{Initial manipulations} \hfill

We start by writing $I_3(\mu,\nu)$, given in \eqref{eq:defI3}, in a form which relates $I_3(\mu,\nu)$ to $I_1(\mu,\nu)$. We have 
\begin{equation*}
    I_3(\mu,\nu) = -\frac{1}{2\pi}\int_1^T\frac{\zeta'}{\zeta}(1-c+it)\zeta^{(\mu)}(1-c+it)\zeta^{(\nu)}(c-it)\ dt.
\end{equation*}

By \cite[Lemma 6]{Gonek1984} we have (where $t\ne 0$)
\begin{equation}\label{logfe2}
    \frac{\zeta'}{\zeta}(1-s)=-\log\left(\frac{|t|}{2\pi}\right)-\frac{\zeta'}{\zeta}(s)+O\left(\frac{1}{1+|t|}\right).
\end{equation}
so
\begin{align}
    I_3(\mu,\nu) &= \frac{1}{2\pi}\int_1^T\left(\log\left(\frac{t}{2\pi}\right)+\frac{\zeta'}{\zeta}(c-it)+O\left(\frac{1}{1+|t|}\right)\right)\zeta^{(\mu)}(1-c+it)\zeta^{(\nu)}(c-it)\ dt \notag\\
    &=\overline{J(\mu,\nu)} + \overline{I_1(\nu,\mu)}+O\left( T^{\frac{1}{2}+\varepsilon} \right), \label{eq:I_3}
\end{align}
where the overline denotes complex conjugate, and where we write for $s=c+it$
\begin{equation*}
    J(\mu,\nu)=\frac{1}{2\pi}\int_1^T\left(\log \frac{t}{2\pi}\right)\zeta^{(\mu)}(1-s)\zeta^{(\nu)}(s)\ dt.
\end{equation*}

We have already evaluated $I_1(\mu,\nu)$ (although note that $\mu$ and $\nu$ are switched here). It therefore only remains to compute $J(\mu,\nu)$. This follows a similar approach, the difference being only minor technical details (in particular, we have a logarithm term in our integrand in place of the logarithmic derivative of zeta in $I_1(\mu,\nu)$).

\begin{remark*}
    We will evaluate $J(\mu,\nu)$ explicitly in terms of the Stieltjes coefficients to be consistent with our approach for $I_1(\mu,\nu)$. However, an alternative approach would be to adapt Ingham's approach \cite{Ing26}, which yields
    \begin{equation*}
    J(\mu,\nu) = \int_1^T \left(\log \frac{t}{2\pi} \right) \left. \frac{\partial^{\mu+\nu}}{\partial \alpha^{\mu+\nu}} \left( \zeta(1+\alpha) + \left(\frac{t}{2\pi}\right)^{-\alpha} \zeta(1-\alpha) \right) \right|_{\alpha=0}   dt  + O\left(T^{1/2+\epsilon}\right) .
    \end{equation*}
\end{remark*}

\subsection{Initial Manipulations of $J(\mu,\nu)$} \hfill

We start by using Lemma \ref{logderiv} on $\zeta^{(\mu)}(1-s)$ to write $J(\mu,\nu)$ as 
\begin{equation*}
   \frac{(-1)^{\mu}}{2\pi}\sum_{k=0}^{\mu}  \binom{\mu}{k}  \int_1^T\zeta^{(\nu)}(c+it)\zeta^{(k)}(c+it)\chi(1-c-it)\left(\log\frac{t}{2\pi}\right)^{\mu-k+1}\ dt+O\left(T^{\frac{1}{2}+\varepsilon}\right).
\end{equation*}

Since we are in the region $\Re(s)=c>1$ we can make use of the Dirichlet series representations of these terms, namely
\begin{equation*}
    \zeta^{(\nu)}(s)\zeta^{(k)}(s)=\sum_{n=1}^{\infty}\frac{B_n^{(\nu,k)}}{n^s}
\end{equation*}
where 
\begin{equation*}
    B_n^{(\nu,k)}=(-1)^{\nu+k}\sum_{n_1n_2=n}(\log n_1)^{\nu}(\log n_2)^{k}
\end{equation*}
so that
\begin{equation*}
    J(\mu,\nu)=\frac{(-1)^{\mu}}{2\pi}\sum_{k=0}^{\mu} \binom{\mu}{k}  \int_1^T\chi(1-c-it)\left(\log\frac{t}{2\pi}\right)^{\mu-k+1}\sum_{n=1}^{\infty}\frac{B_n^{(\nu,k)}}{n^{c+it}} \ dt +O\left(T^{\frac{1}{2}+\varepsilon}\right).
\end{equation*}

Next, by applying Lemma \ref{5.1.2} we have 
\begin{equation} \label{eq:JInTermsOfB}
    J(\mu,\nu)=(-1)^{\mu}\sum_{k=0}^{\mu} \binom{\mu}{k} \sum_{n\le\frac{T}{2\pi}}B_n^{(\nu,k)}(\log n)^{\mu-k+1}+O\left(T^{\frac{1}{2}+\varepsilon}\right).
\end{equation}

\subsection{Evaluating the inner sum} \hfill 

Throughout this section we  we set $Y = \frac{T}{2\pi}$ for notational convenience.

We have the following lemmas, closely resembling Lemmas \ref{lem11DMVT} and \ref{lem12DMVT}. The key difference here is that we don't have a $\zeta'/\zeta(s)$ term (or, equivalently, we don't need to perform a sum over primes), meaning we don't need to split the error term depending on whether we assume the Riemann Hypothesis or not, as we do in Lemma~\ref{lem12DMVT}.
\begin{lemma}\label{residue} 
As $Y\to\infty$ we have
\begin{equation}\label{eqn: E2}
    \sum_{n\le Y}B_n^{(\nu,k)}=\res_{s=1}\left(\zeta^{(\nu)}(s)\zeta^{(k)}(s)\frac{Y^s}{s}\right) + O\left( Y^{\frac{1}{2}+\varepsilon}\right).
\end{equation}
\end{lemma}

\begin{proof}
As in Lemma~\ref{lem11DMVT} we have for  $2\le V\le Y$, as $Y\to\infty$ we have 
\begin{equation*}
\sum_{n\le Y}B_n^{(\nu,k)}=\frac{1}{2\pi i}\int_{c-iV}^{c+iV}\zeta^{(\nu)}(s)\zeta^{(k)}(s)\frac{Y^s}{s}ds+O \left( \frac{Y}{V}(\log Y)^{\nu+k+2} \right).
\end{equation*}

Similar to Lemma~\ref{lem12DMVT}, we shift the contour to the vertical line $c'=1/2$, noting that the integrand has no singularities other than at $s=1$. On the two horizontal pieces $\sigma \pm iV$ we use the unconditional convexity bounds for $\zeta^{(\mu)}(\sigma+iV)$ given in \eqref{eq:convexity}, giving an error of size
\[
\int_{1/2}^1 V^{(1-\sigma)} (\log V)^{\mu+\nu+2} \frac{Y^\sigma}{V} d\sigma \ll \frac{Y}{V} V^\epsilon .
\]
The vertical contour on the line $c'=1/2$ is bounded by
\begin{multline*}
\ll \int_1^V \left|\zeta^{(\mu)}(\tfrac12+it)\right| \left|\zeta^{(\mu)}(\tfrac12+it)\right| \frac{Y^{1/2}}{t} dt \\
\ll Y^{1/2} \left( \int_1^V \frac{\left|\zeta^{(\mu)}(\tfrac12+it)\right|^2 }{t} \right)^{1/2} \left( \int_1^V \frac{\left|\zeta^{(\nu)}(\tfrac12+it)\right|^2 }{t} \right)^{1/2} 
\ll Y^{1/2} V^{\epsilon}
\end{multline*}
the last inequality following from Ingham's \cite{Ing26} evaluation of the second moment of derivatives of zeta on the critical line. Picking any $V$ satisfying $\sqrt{Y}\leq V \leq Y$ produces the desired error, and proves the lemma.
\end{proof}

Evaluating the residue in \eqref{eqn: E2} can be done in a similar way to Lemma \ref{bccoeffs}.

\begin{lemma}\label{lem:dcoeffs}
    We have \begin{equation*}
        \res_{s=1}\left(\zeta^{(\nu)}(s)\zeta^{(k)}(s)\frac{Y^s}{s}\right)=Y\sum_{j=0}^{\nu+k+1}\frac{d_{j}^{(\nu,k)} }{(\nu+k+1-j)!} (\log Y)^{\nu+k+1-j}
    \end{equation*}
    where $d_{j}^{(\nu,k)}$ are the Laurent series coefficients around $s=1$ of 
    \begin{equation*}
        \zeta^{(\nu)}(s)\zeta^{(k)}(s)\frac{1}{s}=\sum_{j\ge 0}d_{j}^{(\nu,k)}(s-1)^{-\nu-k-2+j}.
    \end{equation*}
\end{lemma}

By combining Lemma~\ref{residue} and \ref{lem:dcoeffs} we have \begin{equation*}
    \sum_{n\le Y}B_n^{(\nu,k)}=Y\sum_{j=0}^{\nu+k+1}\frac{d_{j}^{(\nu,k)} }{(\nu+k+1-j)!}(\log Y)^{\nu+k+1-j}+ O\left( Y^{\frac{1}{2}+\varepsilon}\right).
\end{equation*}

The next lemma is analogous to what was proved in  Lemma \ref{lemma_messy_binomial}, where we reinsert the logarithm in the sum above.

\begin{lemma}\label{messy_binomial_2}
    For $k=0,1,\dots,\mu$ we have that
    \begin{equation*}
        \sum_{n\le Y}B_n^{(\nu,k)}(\log n)^{\mu-k+1} = Y\sum_{m=0}^{\mu+\nu+2}C_2^{(\mu,\nu)}(m,k)(\log Y)^m + O\left( Y^{\frac{1}{2}+\varepsilon}\right),    
\end{equation*}
where for $m \geq \mu-k+1$,
\begin{multline*}
    C_2^{(\mu,\nu)}(m,k)=  \frac{(-1)^m(\mu-k+1)}{m!}\sum_{j=0}^{\mu+\nu+1-m} (-1)^{\mu+\nu-j} \frac{(\mu+\nu+1-j)!}{(\nu+k+1-j)!} d_{j}^{(\nu,k)} \\
    +\frac{d_{\mu+\nu+2-m}^{(\nu,k)}}{(k+m-\mu-1)!}
\end{multline*}
and for $m \leq \mu-k$,
\begin{equation*}
    C_2^{(\mu,\nu)}(m,k)=  \frac{(-1)^m(\mu-k+1)}{m!}\sum_{j=0}^{\nu+k+1} (-1)^{\mu+\nu-j} \frac{(\mu+\nu+1-j)!}{(\nu+k+1-j)!} d_{j}^{(\nu,k)}
\end{equation*}
\end{lemma}

\begin{remark*}
In the case when $m=\mu+\nu+2$ (which is the largest $m$ can be), the sum in $C_2^{(\mu,\nu)}(m,k)$ is empty.
\end{remark*}

\begin{proof}
As with the right hand vertical segment in Lemma~\ref{lemma_messy_binomial}, we use partial summation to reinsert the logarithmic term. 
By the partial summation formula we have \begin{multline*}
    \sum_{n\le Y}B_n^{(\nu,k)}(\log n)^{\mu-k+1}=\left(\sum_{n\le Y}B_n^{(\nu,k)}\right)(\log Y)^{\mu-k+1}\\-(\mu-k+1)\int_1^Y\left(\sum_{n\le t}B_n^{(\nu,k)}\right)\frac{(\log t)^{\mu-k}}{t}\ dt.
\end{multline*}   
    
  The first term is
\begin{align*}
\left(\sum_{n\le Y}B_n^{(\nu,k)}\right)(\log Y)^{\mu-k+1} &= Y \sum_{j=0}^{\nu+k+1} \frac{d_{j}^{(\nu,k)}}{(\nu+k+1-j)!} (\log Y)^{\mu+\nu+2-j} + O\left( Y^{\frac{1}{2}+\varepsilon}\right) \\
&= Y \sum_{m=\mu-k+1}^{\mu+\nu+2} \frac{d_{\mu+\nu+2-m}^{(\nu,k)}}{(m+k-\mu-1)!} (\log Y)^{m} + O\left( Y^{\frac{1}{2}+\varepsilon}\right)
\end{align*}
and the second is equal to

\begin{equation*}
-(\mu-k+1)\sum_{j=0}^{\nu+k+1}\frac{d_{j}^{(\nu,k)}}{(\nu+k+1-j)!} \int_1^Y (\log t)^{\mu+\nu+1-j} dt  +  O\left( Y^{\frac{1}{2}+\varepsilon}\right)
\end{equation*}
and using \eqref{usefulintegral}, this evaluates to
\begin{equation*}
    -(\mu-k+1)\sum_{j=0}^{\nu+k+1}\frac{d_{j}^{(\nu,k)}}{(\nu+k+1-j)!}(-1)^{\mu+\nu+1-j}(\mu+\nu+1-j)!Y\sum_{m=0}^{\mu+\nu+1-j}\frac{(-1)^{m}(\log Y)^{m}}{m!}
\end{equation*}
Swapping the order of summation, the double sum is
\begin{multline*}
(\mu-k+1) Y \sum_{m=0}^{\mu-k} \frac{(-1)^{m}(\log Y)^{m}}{m!}   \sum_{j=0}^{\nu+k+1} (-1)^{\mu+\nu-j} \frac{(\mu+\nu+1-j)!}{(\nu+k+1-j)!} d_{j}^{(\nu,k)}  \\
+(\mu-k+1) Y \sum_{m=\mu-k+1}^{\mu+\nu+1} \frac{(-1)^{m}(\log Y)^{m}}{m!}   \sum_{j=0}^{\mu+\nu+1-m} (-1)^{\mu+\nu-j} \frac{(\mu+\nu+1-j)!}{(\nu+k+1-j)!} d_{j}^{(\nu,k)}.
\end{multline*}
Combining the two pieces gives the formula for $C_2^{(\mu,\nu)}(m,k)$ and completes the proof.
\end{proof}

Applying Lemma \ref{messy_binomial_2} to \eqref{eq:JInTermsOfB} we have
\begin{equation}\label{J1poly_simplified}
    J(\mu,\nu)=\frac{T}{2\pi} \sum_{m=0}^{\mu+\nu+2} \sum_{k=0}^{\mu} (-1)^{\mu} \binom{\mu}{k}   C_2^{(\mu,\nu)}(m,k) \left(\log\frac{T}{2\pi}\right)^m + O\left( T^{\frac{1}{2}+\varepsilon}\right).
\end{equation}

Switching the role of $\mu$ and $\nu$ in \eqref{I1poly_simplified}, we have that
\begin{equation}\label{I1_reversepoly_simplified}
    I_1(\nu,\mu)=\frac{T}{2\pi} \sum_{m=0}^{\mu+\nu+2} \sum_{k=0}^{\mu} (-1)^{\mu} \binom{\mu}{k} C_1^{(\nu,\mu)}(m,k) \left(\log\frac{T}{2\pi}\right)^m + O\left(Te^{-C\sqrt{\log T}}\right)
\end{equation}
so combining $\overline{J(\mu,\nu)}$ and $\overline{I_1(\nu,\mu)}$ in \eqref{eq:I_3} gives
\begin{multline}\label{Jpoly}
    I_3(\mu,\nu)=\frac{T}{2\pi} \sum_{m=0}^{\mu+\nu+2} \sum_{k=0}^{\mu} (-1)^{\mu} \binom{\mu}{k} \left(C_1^{(\nu,\mu)}(m,k)+C_2^{(\mu,\nu)}(m,k)\right) \left(\log\frac{T}{2\pi}\right)^m\\
    +O\left(Te^{-C\sqrt{\log T}}\right)
\end{multline}

Substituting \eqref{I1poly_simplified} and \eqref{Jpoly} into \eqref{eq:I_I1_I3}, we deduce that
\begin{multline*}
    I(\mu,\nu)=\frac{T}{2\pi} \sum_{m=0}^{\mu+\nu+2} \sum_{k=0}^{\nu} (-1)^{\nu} \binom{\nu}{k} C_1^{(\mu,\nu)}(m,k) \left(\log\frac{T}{2\pi}\right)^m\\ 
    +\frac{T}{2\pi} \sum_{m=0}^{\mu+\nu+2} \sum_{k=0}^{\mu} (-1)^{\mu} \binom{\mu}{k}  \left(C_1^{(\nu,\mu)}(m,k)+C_2^{(\mu,\nu)}(m,k)\right) \left(\log\frac{T}{2\pi}\right)^m + O\left(Te^{-C\sqrt{\log T}}\right),
\end{multline*}
completing the proof of Theorem \ref{mainthm}.

\section{Proof of the Corollaries}\label{conclusion}

As already noted, Corollary \ref{cor:higher_derivs_milinovich} follows immediately from Theorem \ref{mainthm}.

\begin{proof}[Proof of Corollary \ref{cor:micah}]

We show that our theorem with $\mu=\nu=1$ recovers precisely Milinovich's full asymptotic expansion for the first derivative, stated in \eqref{milinovich_theorem}.

Theorem~\ref{mainthm} states that, under the Riemann Hypothesis,
\begin{multline*}
\sum_{0<\gamma\le T}\left|\zeta'\left(\frac{1}{2}+i\gamma\right)\right|^2 \\
= \frac{T}{2\pi} \sum_{m=0}^4 (-2C_1^{(1,1)}(m,0) -C_2^{(1,1)}(m,0) -  2C_1^{(1,1)}(m,1) - C_2^{(1,1)}(m,1)) \left(\log\frac{T}{2\pi}\right)^m \\
+ O\left(T^{\frac{1}{2}+\varepsilon}\right).
\end{multline*}

We now calculate the various terms in this expansion. We have
\[
C_1^{(1,1)}(m,0) = 
\begin{cases}
     c^{(1,0)}_0-c^{(1,0)}_1+c^{(1,0)}_2-c^{(1,0)}_3 & \text{ if } m=0 \\[1.5ex]
    -c^{(1,0)}_0+c^{(1,0)}_1-c^{(1,0)}_2+c^{(1,0)}_3 & \text{ if } m=1 \\[1.5ex]
   \frac{1}{2} c^{(1,0)}_0- \frac{1}{2}  c^{(1,0)}_1 +c^{(1,0)}_2 & \text{ if } m=2 \\[1.5ex]
    -\frac{1}{6} c^{(1,0)}_0 +\frac{1}{2} c^{(1,0)}_1 & \text{ if } m=3\\[1.5ex]
    \frac{1}{6} c^{(1,0)}_0 & \text{ if } m=4
\end{cases}
\]
where $c^{(1,0)}_j$ are the coefficients of
\[
\frac{\zeta'(s)^2}{s} = \frac{c^{(1,0)}_0}{(s-1)^4} + \frac{c^{(1,0)}_1}{(s-1)^3} + \frac{c^{(1,0)}_2}{(s-1)^2} + \frac{c^{(1,0)}_3}{(s-1)} + \dots
\]
and where
\[
C_2^{(1,1)}(m,0) = 
\begin{cases}
6 d^{(1,0)}_0 -4 d^{(1,0)}_1 + 2d^{(1,0)}_2 & \text{ if } m=0 \\[1.5ex]
-6 d^{(1,0)}_0 + 4d^{(1,0)}_1 - 2d^{(1,0)}_2 & \text{ if } m=1 \\[1.5ex]
3 d^{(1,0)}_0 -2d^{(1,0)}_1 +d^{(1,0)}_2& \text{ if } m=2 \\[1.5ex]
-d^{(1,0)}_0 + d^{(1,0)}_1& \text{ if } m=3\\[1.5ex]
\frac{1}{2} d^{(1,0)}_0 & \text{ if } m=4
\end{cases}
\]
where $d^{(1,0)}_j$ are the coefficients of
\[
\frac{\zeta'(s) \zeta(s)}{s} = \frac{d^{(1,0)}_0}{(s-1)^3} + \frac{d^{(1,0)}_1}{(s-1)^2} + \frac{d^{(1,0)}_2}{(s-1)} + \dots
\]
and where
\[
C_1^{(1,1)}(m,1) = 
\begin{cases}
c^{(1,1)}_4 & \text{ if } m=0 \\[1.5ex]
c^{(1,1)}_3 & \text{ if } m=1 \\[1.5ex]
\frac{1}{2} c^{(1,1)}_2 & \text{ if } m=2 \\[1.5ex]
\frac{1}{6} c^{(1,1)}_1 & \text{ if } m=3\\[1.5ex]
\frac{1}{24} c^{(1,1)}_0 & \text{ if } m=4
\end{cases}
\]
where $c^{(1,1)}_j$ are the coefficients of
\[
\frac{\zeta'(s)^3}{\zeta(s)} \frac{1}{ s} = \frac{c^{(1,1)}_0}{(s-1)^5} + \frac{c^{(1,1)}_1}{(s-1)^4} + \frac{c^{(1,1)}_2}{(s-1)^3} + \frac{c^{(1,1)}_3}{(s-1)^2} + \frac{c^{(1,1)}_4}{(s-1)} + \dots
\]
and where
\[
C_2^{(1,1)}(m,1) = 
\begin{cases}
d^{(1,1)}_0 -d^{(1,1)}_1 +d^{(1,1)}_2 -d^{(1,1)}_3  & \text{ if } m=0 \\[1.5ex]
-d^{(1,1)}_0 +d^{(1,1)}_1 -d^{(1,1)}_2 +d^{(1,1)}_3  & \text{ if } m=1 \\[1.5ex]
\frac{1}{2} d^{(1,1)}_0 -\frac{1}{2} d^{(1,1)}_1 +d^{(1,1)}_2 & \text{ if } m=2 \\[1.5ex]
 -\frac{1}{6} d^{(1,1)}_0 +\frac{1}{2} d^{(1,1)}_1 & \text{ if } m=3 \\[1.5ex]
\frac{1}{6} d^{(1,1)}_0 & \text{ if } m=4 \\[1.5ex]
\end{cases}
\]
where $d^{(1,1)}_j$ are the coefficients of
\[
\frac{\zeta'(s)^2}{s} = \frac{d^{(1,1)}_0}{(s-1)^4} + \frac{d^{(1,1)}_1}{(s-1)^3} + \frac{d^{(1,1)}_2}{(s-1)^2} + \frac{d^{(1,1)}_3}{(s-1)} + \dots
\]
which we note are the same as $c^{(1,0)}_j$ in this special case.

Evaluating the Laurent coefficients $c^{(1,k)}_j$ and $d^{(1,k)}_j$ and inserting them into $C_1(m,k)$ and $C_2(m,k)$ yields \eqref{milinovich_theorem}.
\end{proof}

\begin{proof}[Proof of Corollary \ref{corollary:general}]

By explicitly evaluating the highest-order term in the polynomial in Theorem~\ref{mainthm} we can recover the leading order coefficient, first found by Gonek.

    By Theorem \ref{mainthm} we have that the leading order is given by taking the $m=\mu+\nu+2$ term in \eqref{eq:poly}, that is,
    \begin{multline}\label{eq:leading}
        \sum_{k=0}^{\nu}(-1)^{\nu}\binom{\nu}{k}C_1^{(\mu,\nu)}(\mu+\nu+2,k)\\+\sum_{k=0}^{\mu}(-1)^{\mu}\binom{\mu}{k}\left(C_1^{(\nu,\mu)}(\mu+\nu+2,k)+C_2^{(\mu,\nu)}(\mu+\nu+2,k)\right).
    \end{multline}

    At the very top power , the sum over $j$ pieces in $C_1, C_2$ are empty and so only the first terms remain,
    \begin{align*}
        C_1^{(\mu,\nu)}(\mu+\nu+2,k)&=\frac{c_0^{(\mu,k)}}{(k+\mu+2)!}=\frac{(-1)^{\mu+k+1}\mu!k!}{(k+\mu+2)!}\\ 
        C_1^{(\nu,\mu)}(\mu+\nu+2,k)&=\frac{c_0^{(\nu,k)}}{(k+\nu+2)!}=\frac{(-1)^{\nu+k+1}\nu!k!}{(k+\nu+2)!}\\
        C_2^{(\mu,\nu)}(\mu+\nu+2,k)&=\frac{d_0^{(\nu,k)}}{(k+\nu+1)!}=\frac{(-1)^{\nu+k}\nu!k!}{(k+\nu+1)!}.
    \end{align*}

Substituting these values into \eqref{eq:leading} gives
    \begin{equation*}
        \sum_{k=0}^{\nu}(-1)^{\nu}\binom{\nu}{k}\frac{(-1)^{\mu+k+1}\mu!k!}{(k+\mu+2)!}+\sum_{k=0}^{\mu}(-1)^{\mu}\binom{\mu}{k}\left(\frac{(-1)^{\nu+k+1}\nu!k!}{(k+\nu+2)!}+\frac{(-1)^{\nu+k}\nu!k!}{(k+\nu+1)!}\right)
    \end{equation*}
and evaluating these sums completes the proof of Corollary \ref{corollary:general}.
\end{proof}

\appendix

\section{The second derivative}\label{app:2ndDeriv}

Employing Mathematica, we evaluated the full asymptotic polynomial for the second derivative. This demonstrates the complexity when writing the result out in full.

We have
\[
\sum_{0<\gamma\le T}\left|\zeta''\left(\frac{1}{2}+i\gamma\right)\right|^2 = \frac{T}{2\pi} \sum_{m=0}^6 A_j \left( \log\frac{T}{2\pi}\right)^j
\]
where the coefficients $A_j$ are given in Table~\ref{tab:SecondDeriv}.
\begin{table}[hp]
    \begin{tabular}{|l|p{0.8\textwidth}|}
    \hline
    &\\
    $A_6$ & $\frac{4}{45}$ \\[2ex]
    $A_5$ & $-\frac{8}{15} + \frac{11 \gamma_0 }{15}$ \\[2ex]
    $A_4$ & $\frac{8}{3}-\frac{11 \gamma_0 }{3}+\gamma_0 ^2 -\frac{8 \gamma _1}{3}$ \\[2ex]
    $A_3$ & $-\frac{32}{3}+\frac{44 \gamma_0 }{3}-4 \gamma_0 ^2-\frac{8
   \gamma_0 ^3}{3} +  \frac{32 \gamma _1}{3}-12 \gamma_0  \gamma _1+\frac{2 \gamma
   _2}{3} $ \\[2ex]
    $A_2$ & $32-44 \gamma_0 +12 \gamma_0 ^2+8 \gamma_0
   ^3+4 \gamma_0 ^4  -32 \gamma _1+36 \gamma_0  \gamma _1+24 \gamma_0 ^2 \gamma _1+24
   \gamma _1^2-2 \gamma _2+16 \gamma_0  \gamma
   _2+\frac{8 \gamma _3}{3}$ \\[2ex]
    $A_1$ & $ -64+88 \gamma_0
   -24 \gamma_0 ^2-16 \gamma_0 ^3-8 \gamma_0 ^4  +64 \gamma _1-72 \gamma_0  \gamma _1-48 \gamma_0 ^2 \gamma _1-8 \gamma_0
   ^3 \gamma _1-48 \gamma _1^2-24 \gamma_0 
   \gamma _1^2+4 \gamma _2-32 \gamma_0  \gamma _2-12
   \gamma_0 ^2 \gamma _2-32 \gamma _1 \gamma _2 -\frac{16 \gamma
   _3}{3}-8 \gamma_0  \gamma _3-\frac{4 \gamma _4}{3}$ \\[2ex]
    $A_0$ & $64-88 \gamma_0 +24 \gamma_0 ^2+16 \gamma_0 ^3+8 \gamma_0 ^4-8
   \gamma_0 ^6 -64 \gamma _1+72 \gamma_0  \gamma _1+48 \gamma_0 ^2 \gamma _1+8
   \gamma_0 ^3 \gamma _1-48 \gamma_0 ^4 \gamma _1+48 \gamma
   _1^2+24 \gamma_0  \gamma _1^2-72 \gamma_0
   ^2 \gamma _1^2-16 \gamma _1^3
   -4 \gamma _2+32 \gamma_0  \gamma _2+12 \gamma_0 ^2 \gamma _2-16
   \gamma_0 ^3 \gamma _2+32 \gamma _1 \gamma _2-24 \gamma_0  \gamma
   _1 \gamma _2+4 \gamma _2^2+\frac{16 \gamma
   _3}{3}+8 \gamma_0  \gamma _3+8 \gamma _1 \gamma _3
   +\frac{4\gamma _4}{3}
   +2 \gamma_0  \gamma _4+\frac{14 \gamma
   _5}{15}$\\[2ex]
    \hline
    \end{tabular}
    \caption{The coefficients of the asymptotic degree-6 polynomial for the absolute value squared of the second derivative, $ \left|\zeta''\left(\frac{1}{2}+i\gamma\right)\right|^2$.}\label{tab:SecondDeriv}
\end{table}

Similarly, we have
\[
\sum_{0<\gamma\le T} \zeta'\left(\frac{1}{2}+i\gamma\right) \zeta''\left(\frac{1}{2}-i\gamma\right) = \frac{T}{2\pi} \sum_{m=0}^5 B_j \left( \log\frac{T}{2\pi}\right)^j
\]
where the coefficients $B_j$ are given in Table~\ref{tab:MixedDeriv}.
\begin{table}[hp]
    \begin{tabular}{|l|p{0.8\textwidth}|}
    \hline
    &\\
    $B_5$ & $-\frac{1}{12}$ \\[2ex]
    $B_4$ & $\frac{5}{12}-\frac{2 \gamma_0 }{3}$ \\[2ex]
    $B_3$ & $-\frac{5}{3}+\frac{8 \gamma_0 }{3}-\gamma_0^2 + 2 \gamma_1$ \\[2ex]
    $B_2$ & $5-8 \gamma_0 +3 \gamma_0^2+2 \gamma_0^3 -6 \gamma_1+10 \gamma_0  \gamma_1+\gamma_2$ \\[2ex]
    $B_1$ & $-10+16 \gamma_0 -6 \gamma_0^2-4 \gamma_0^3-2 \gamma_0^4 + 12 \gamma_1-20 \gamma_0  \gamma_1-12 \gamma_0^2 \gamma_1-14 \gamma_1^2-2 \gamma_2-8 \gamma_0  \gamma_2-\frac{10 \gamma_3}{3}$ \\[2ex]
    $B_0$ & $10-16 \gamma_0 +6 \gamma_0^2+4 \gamma_0^3+2 \gamma_0^4 -12 \gamma_1+20 \gamma_0  \gamma_1+12 \gamma_0^2 \gamma_1+14 \gamma_1^2+2 \gamma_2+8 \gamma_0  \gamma_2+\frac{10 \gamma_3}{3}$\\[2ex]
    \hline
    \end{tabular}
    \caption{The coefficients of the asymptotic degree-5 polynomial for the mixed first and second derivative, $\zeta'\left(\frac{1}{2}+i\gamma\right) \zeta''\left(\frac{1}{2}-i\gamma\right)$.}\label{tab:MixedDeriv}
\end{table}

\section{Graphical illustration}\label{app:graphs}

We illustrate the power of our result by plotting the graphs for $\mu=\nu=1$ in Figure~\ref{fig1} and $\mu=\nu=2$ in Figure~\ref{fig2}, in each figure showing the full sum
\[
\sum_{0<\gamma\le T}\left|\zeta^{(\nu)} \left(\frac{1}{2}+i\gamma\right)\right|^2 ,
\]
the sum minus the leading order asymptotic term, and the sum minus the full asymptotic as given in Theorem~\ref{mainthm}. Even though we only plot this over the first 100,000 zeros, the graphs clearly demonstrate the power of the result.

Similarly in Figure~\ref{fig3} we plot the graph for $\mu=1$ and $\nu=2$. The true sum is not real, but the imaginary parts are very small compared to the real parts, and we show 
\[
\Im \sum_{0<\gamma\le T} \zeta'\left(\frac{1}{2}+i\gamma\right) \zeta''\left(\frac{1}{2}-i\gamma\right) 
\]
as well as the error coming from subtracting the full asymptotic from the real part of the sum.

\begin{figure}[ht]
		\begin{subfigure}{.5\textwidth}
			\centering
			\includegraphics[width=.8\linewidth]{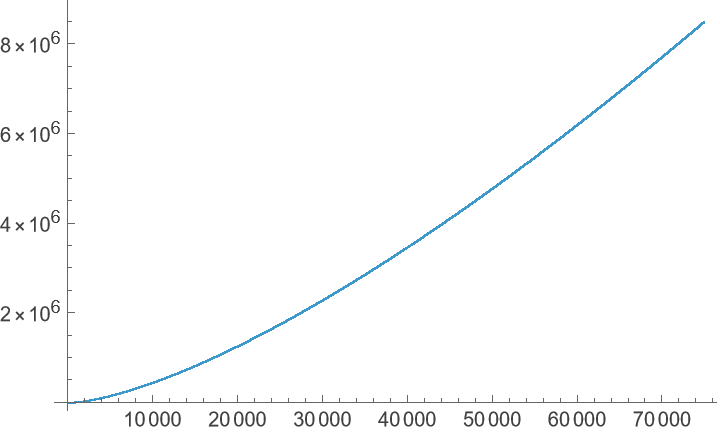}
			\caption{$\sum_{0<\gamma<T} \left|\zeta'\left(\frac{1}{2}+i\gamma\right)\right|^2$}
			\label{fig1:Truth}
		\end{subfigure}%
		\begin{subfigure}{.5\textwidth}
			\centering
			\includegraphics[width=.8\linewidth]{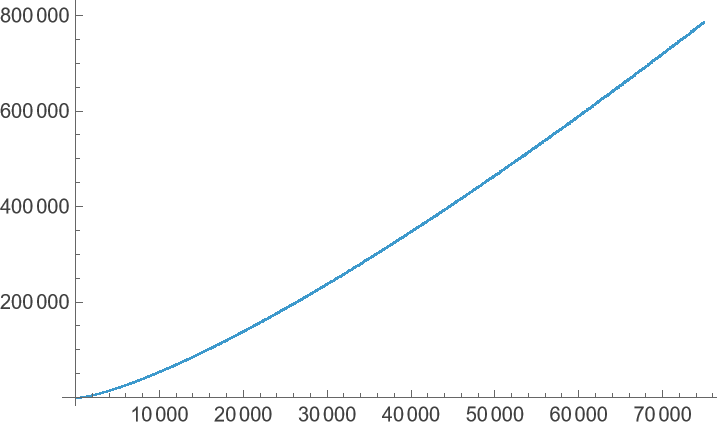}
			\caption{$\sum_{0<\gamma<T} \left|\zeta'\left(\frac{1}{2}+i\gamma\right)\right|^2 - \frac{1}{24\pi}  T \left(\log \frac{T}{2\pi}\right)^4$}
			\label{fig1:LeadingOrder}
		\end{subfigure}\\
		\begin{subfigure}{.5\textwidth}
			\centering
			\includegraphics[width=.8\linewidth]{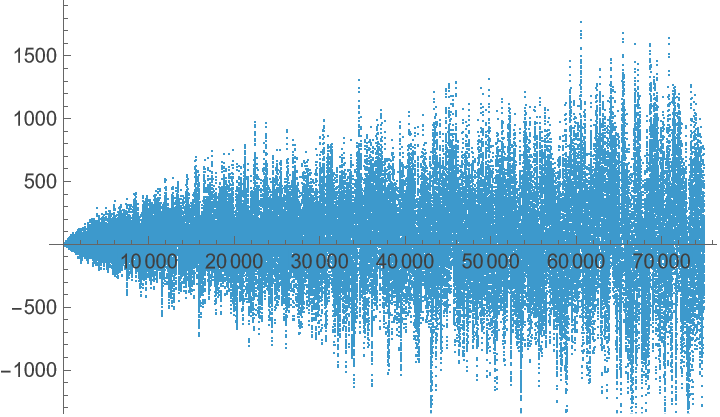}
			\caption{$\sum_{0<\gamma<T} \left|\zeta'\left(\frac{1}{2}+i\gamma\right)\right|^2$ minus the full asymptotic given in \eqref{milinovich_theorem}.}
			\label{fig1:FullAsymptotic}
		\end{subfigure}%
		\caption{The first derivative}
		\label{fig1}
	\end{figure}

\begin{figure}[ht]
		\begin{subfigure}{.5\textwidth}
			\centering
			\includegraphics[width=.8\linewidth]{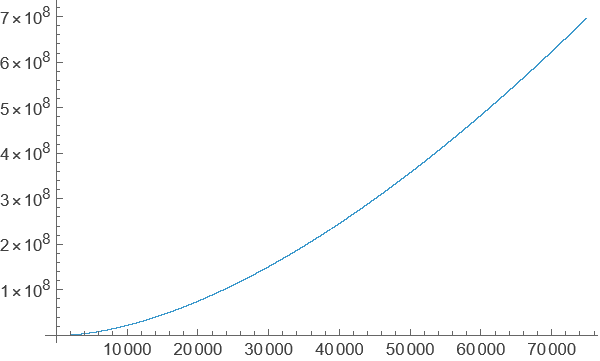}
			\caption{$\sum_{0<\gamma<T} \left|\zeta''\left(\frac{1}{2}+i\gamma\right)\right|^2$}
			\label{fig2:Truth}
		\end{subfigure}%
		\begin{subfigure}{.5\textwidth}
			\centering
			\includegraphics[width=.8\linewidth]{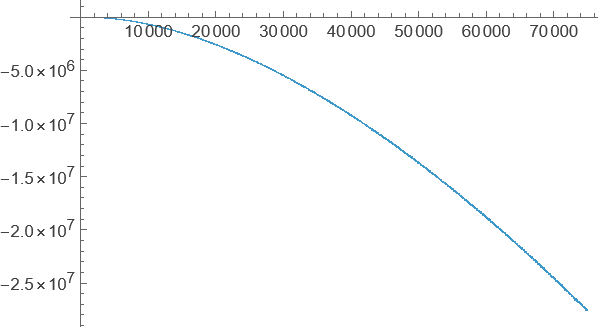}
			\caption{$\sum_{0<\gamma<T} \left|\zeta''\left(\frac{1}{2}+i\gamma\right)\right|^2 - \frac{4}{90\pi} T \left(\log \frac{T}{2\pi}\right)^6$}
			\label{fig2:LeadingOrder}
		\end{subfigure}\\
		\begin{subfigure}{.5\textwidth}
			\centering
			\includegraphics[width=.8\linewidth]{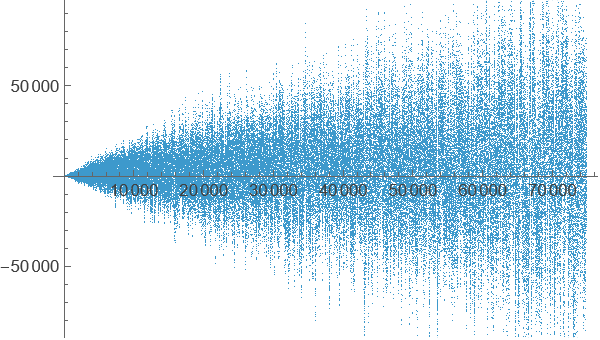}
			\caption{$\sum_{0<\gamma<T} \left|\zeta''\left(\frac{1}{2}+i\gamma\right)\right|^2$  minus the full asymptotic given in Appendix~\ref{app:2ndDeriv}}
			\label{fig2:FullAsymptotic}
		\end{subfigure}%
		\caption{The second derivative}
		\label{fig2}
	\end{figure}

\begin{figure}[ht]
		\begin{subfigure}{.5\textwidth}
			\centering
			\includegraphics[width=.8\linewidth]{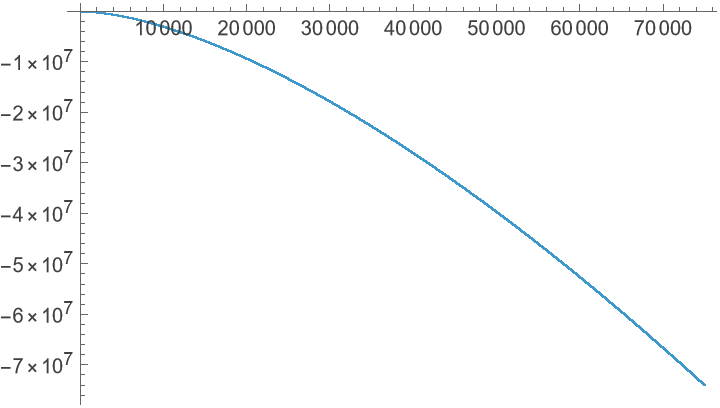}
			\caption{$\Re\sum_{0<\gamma<T}  \zeta'\left(\frac{1}{2}+i\gamma\right) \zeta''\left(\frac{1}{2}-i\gamma\right)$}
			\label{fig3:Truth}
		\end{subfigure}%
		\begin{subfigure}{.5\textwidth}
			\centering
			\includegraphics[width=.8\linewidth]{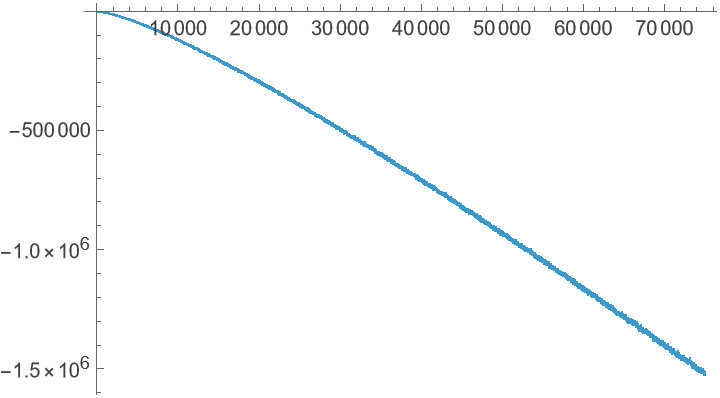}
			\caption{$\Re\sum_{0<\gamma<T}  \zeta'\left(\frac{1}{2}+i\gamma\right) \zeta''\left(\frac{1}{2}-i\gamma\right)- \frac{1}{24\pi} T \left(\log \frac{T}{2\pi}\right)^5$}
			\label{fig3:LeadingOrder}
		\end{subfigure}\\
		\begin{subfigure}{.5\textwidth}
			\centering
			\includegraphics[width=.8\linewidth]{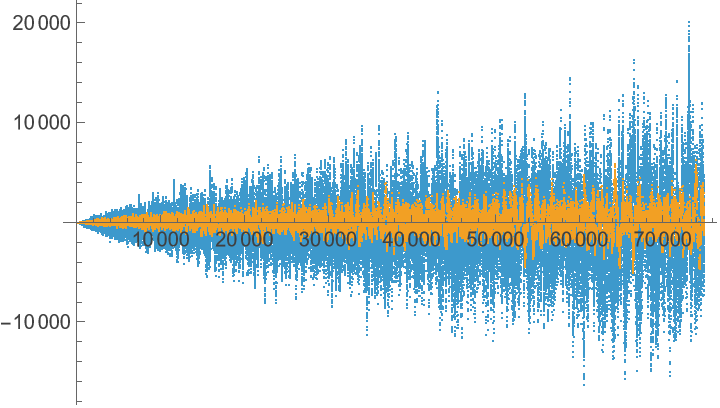}
			\caption{$\Re\sum_{0<\gamma<T}  \zeta'\left(\frac{1}{2}+i\gamma\right) \zeta''\left(\frac{1}{2}-i\gamma\right)$  minus the full asymptotic given in Appendix~\ref{app:2ndDeriv} (blue), and also $\Im\sum_{0<\gamma<T}  \zeta'\left(\frac{1}{2}+i\gamma\right) \zeta''\left(\frac{1}{2}-i\gamma\right)$ (orange).}
			\label{fig3:FullAsymptotic}
		\end{subfigure}%
		\caption{The mixed first and second derivatives}
		\label{fig3}
	\end{figure}

\newpage

\addcontentsline{toc}{chapter}{Bibliography}

\bibliography{references}    
\bibliographystyle{plain}

\end{document}